\documentclass[onefignum,onetabnum]{siamart190516}

\usepackage{amssymb}
\usepackage{bm}
\usepackage{upgreek}
\usepackage{amsmath}
\usepackage{comment}
\usepackage{amsfonts}
\usepackage{graphicx}
\usepackage{epstopdf}
\usepackage{algorithmic}
\usepackage{textcomp}
\usepackage{listings}
\usepackage{hyperref}
\usepackage{subcaption}
\usepackage{tikz}
\usetikzlibrary{arrows}
\usepackage[utf8]{inputenc}
\usepackage{pgfplots}
\usepackage{pbox}
\usepackage{bem}

\pgfplotsset{every tick label/.append style={font=\footnotesize}}
\pgfplotsset{label style={font=\footnotesize}}

\graphicspath{{img/all/}}

\usepackage{amsopn}

\usepackage{cleveref}

\ifpdf
  \DeclareGraphicsExtensions{.eps,.pdf,.png,.jpg}
\else
  \DeclareGraphicsExtensions{.eps}
\fi

\newsiamremark{remark}{Remark}
\newsiamremark{hypothesis}{Hypothesis}
\crefname{hypothesis}{Hypothesis}{Hypotheses}
\newsiamthm{claim}{Claim}

\newsiamremark{assumption}{Assumption}
\crefname{assumption}{assumption}{assumptions}
\newsiamremark{property}{Property}
\crefname{property}{property}{properties}

\newcommand{\plotsize}{0.48}
\newcommand{\eigenval}{l}

\newcommand{\dotplot}[1][]{\addplot [red, mark=*, thick, mark options={solid},mark size=1,#1]}
\newcommand{\oneplot}[1][]{\addplot [red, mark=triangle*, thick, mark options={solid,draw=black},#1]}

\newcommand{\sametwoplot}[1][]{\addplot [red, mark=diamond*, thick, mark options={solid,draw=black},#1]}
\newcommand{\samethreeplot}[1][]{\addplot [red, mark=pentagon*, thick, mark options={solid,draw=black},#1]}
\newcommand{\oneplotdesc}{red triangles}

\newcommand{\sametwoplotdesc}{red diamonds}
\newcommand{\samethreeplotdesc}{red pentagons}

\newcommand{\TODO}[1][]{{\color{red}TODO\ifthenelse{\equal{#1}{}}{}{: #1}}}

\headers{BEM with weakly imposed boundary conditions}{T. Betcke, E. Burman, and M. W. Scroggs}

\title{Boundary element methods for Helmholtz problems with weakly imposed boundary conditions
\thanks{Submitted to the editors 28 April 2020. Second draft submitted to the editors 14 September 2021.}}

\author{Timo Betcke\thanks{Department of Mathematics, University College London, WC1E 6BT, UK
  (\email{t.betcke@ucl.ac.uk}).}
\and Erik Burman\footnotemark[2]\thanks{Department of Mathematics, University College London, WC1E 6BT, UK
  (\email{e.burman@ucl.ac.uk}).}
\and Matthew W. Scroggs\thanks{Department of Engineering, University of Cambridge, CB2 1PZ, UK
  (\email{mws48@cam.ac.uk}, \url{https://www.mscroggs.co.uk}).}}

\ifpdf
\hypersetup{
  pdftitle={Boundary element methods for Helmholtz problems with weakly imposed boundary conditions},
  pdfauthor={T. Betcke, E. Burman, and M. W. Scroggs}
}
\fi

\begin{document}

\maketitle

\begin{abstract}
We consider boundary element methods where the Calder\'on projector is
used for the system matrix and boundary conditions are weakly imposed using a
particular variational boundary operator designed using techniques
from augmented Lagrangian methods. Regardless of the boundary conditions,
both the primal trace variable and the flux are approximated. We focus on the imposition of Dirichlet
conditions on the Helmholtz equation, and extend the analysis of the Laplace problem from \emph{Boundary element methods with weakly imposed boundary conditions}
\cite{BeBuScro18} to this case.
The theory is illustrated by a series of numerical examples.
\end{abstract}

\begin{keywords}
  boundary element methods, Nitsche's method, Helmholtz equation, wave scattering
\end{keywords}

\begin{AMS}
    65N38, 65R20
\end{AMS}

\section{Introduction}\label{sec:intro}
In a previous paper \cite{BeBuScro18}, we introduced a method of weakly imposing boundary conditions on the boundary element method,
inspired by Nitsche's method \cite{Nit71} and Babu\u{s}ka's penalty method \cite{Babuska1973} for the finite element method.
Weak imposition of boundary conditions here means that neither the Dirichlet trace nor the Neumann trace is imposed exactly,
instead an $h$-dependent boundary condition is imposed that is weighted in such a way that optimal error estimates may be derived
and the exact boundary condition is recovered in the asymptotic limit.

In \cite{BeBuScro18}, we introduced the weak imposition of Dirichlet, Neumann and Robin boundary conditions on Laplace's equation;
in \cite{BuFrScro19}, we applied this method to Signorini contact conditions, again for Laplace's equation.
In this paper, we look at how this method and its analysis can be extended to be used for the Helmholtz equation, focussing on the
exterior Helmholtz Dirichlet problem: Find $u=u\inc+u\scat\in\Hlocopspace[\Omega\exterior]{\Delta}{1}$ such that
\begin{subequations}
\label{eq:Helmholtz}
\begin{align}
-\Delta u-k^2u&=0&&\text{in }\Omega\exterior,\label{eq:Helmholtz_pm}\\
\frac{\partial u\scat}{\partial\abs{\x}}-\ii ku\scat&=o(\abs{\x}^{-1})&&\text{as }\abs{\x}\to\infty,
\label{eq:Helmholtz_infty}\\
u&=g\D&&\text{on }\Gamma,\label{eq:Laplace_diribc}
\end{align}
\end{subequations}
where $\Omega\interior\subset\RR^3$ is a bounded interior open domain with polyhedral boundary $\Gamma$,
$\Omega\exterior=\RR^3\setminus\overline{\Omega\interior}$ is the open domain exterior to $\Omega\interior$,
$\vec{\nu}$ is the unit normal to the surface $\Gamma$ pointing outwards into $\Omega\exterior$,
$u\inc$ is a known incident wave,
and $k\in\RR$ is the wavenumber of the problem.
We assume that $g\D\in\Hspace[\Gamma]{1/2}$.
Whenever it is ambiguous, we write $\vec{\nu}_\vec{x}$ to denote the outward-pointing normal at the point $\vec{x}$.

Due to the Sommerfeld radiation condition \cref{eq:Helmholtz_infty}, the problem \cref{eq:Helmholtz} has
a unique solution \cite{Sommerfeld}.
The formulation of Helmholtz problems using boundary integral equations are covered in detail in \cite{Nedelec01},
and their discretisation is examined in \cite[section 7.6]{Stein07}.

The use of discretisation and the boundary element method to solve Helmholtz problems has been well studied.
For sufficiently small wavenumbers $k$, and sufficiently smooth boundaries, the operators involved are coercive,
and hence \emph{a priori} error bounds can be derived \cite{Spence15,Spence11,Graham07}.
For values of $k$ and domains for which coercivity cannot be shown, error bounds have been shown that involve both the mesh size
$h$ and the wavenumber $k$ \cite{Banjai07,Schatz74}. If the wavenumber is varied, then the mesh must be refined to keep the 
value of $hk$ constant in order to maintain a low error \cite{Graham15}.
The use of $hp$-BEM methods for Helmholtz has also been studied and analysed \cite{Melenk11-2,Melenk12}.

The use of blocked operator formulations to solve Helmholtz problems is common for domain decomposition problems,
where the boundary element method is used in multiple domains with different wavenumbers \cite{Sayas08,Sayas008},
or a combination of finite and boundary element methods can be used \cite{Hiptmair06}.
To avoid the appearance of spurious resonances in solutions, coupled stabilised formulations can be solved \cite{Steinbach08}.
The formulations presented in this paper are, in general, more expensive than standard formulations, as they require the
assembly of the full Calder\'on system. In these cases, however, larger blocked systems are already being assembled,
and so it may be possible to impose boundary conditions on them weakly with little additional cost.

The method proposed in this paper is applicable to low and medium frequency problems. In practice, preconditioning limits the method's effectiveness
for higher frequency problems. However, one main advantage of this method is its immunity to eigenvalues of the interior problem:
the solution to \cref{eq:Helmholtz} can be found for any real wavenumber using our method without any modification to stabilise
against eigenvalues being necessary.

Although we do not present an analysis of problems with mixed boundary conditions,
we discuss our method's potential application to such problems in \cref{sec:mixed}.
As we discussed in \cite{BeBuScro18} for Laplace problems, our method of weakly
imposing boundary conditions is advantageous when solving mixed problems, as they can
be implemented by assembling different sparse terms on different parts of the mesh
without needing to adapt the Calder\'on part of the formulation.

In \cref{sec:bem}, we define the boundary operators used in our formulations, and present some of their important properties.
In \cref{sec:derive}, we derive our formulation for Dirichlet Helmholtz problems.
In \cref{sec:analysis}, we analyse this formulation, and prove \emph{a priori} error bounds.
In \cref{sec:numerical}, we present some numerical experiments, and in \cref{sec:conclusions} we give some concluding remarks.

\section{Boundary operators}\label{sec:bem}
We define the Green's function for the Helmholtz operator in $\RR^3$ by
\begin{equation}
G(\x,\y) =
\frac{\e^{\ii k\seminorm{\x-\y}}}{4\uppi|\x-\y|}.\label{def:HelmGreen}
\end{equation}
In this paper, we focus on the problem in
$\RR^3$. Similar analysis can be used for problems in $\RR^2$, in which case this definition should be replaced by
$G(\x,\y) =\frac\ii4 H_0^{(1)}(k\seminorm{\x-\y})$, where $H_0^{(1)}$ is a Hankel function of the first kind.

In the standard fashion (see eg \cite[chapter 6]{Stein07}), we define
the single layer potential operator, $\pop{V}:\Hspace{-1/2}\to \Hlocspace[\Omega\exterior\cup\Omega\interior]{1}$,
and
the double layer potential operator, $\pop{K}:\Hspace{1/2}\to \Hlocspace[\Omega\exterior\cup\Omega\interior]{1}$,
for $v\in \Hspace{1/2}$, $\mu\in \Hspace{-1/2}$, and $\x\in\RR^3\setminus\Gamma$ by
\begin{align}
(\popV\mu)(\x) &:= \int_{\Gamma} G(\x,\y) \mu(\y)\dx[\y]\label{eq:single},\\
(\popK v)(\x) &:= \int_{\Gamma} \frac{\partial G(\x,\y)}{\partial\vec\nu_\y} v(\y)\dx[\y]\label{eq:double}.
\end{align}
We recall that $\vec\nu_\y$ denotes the normal to the surface $\Gamma$ at the point $\y$ pointing outwards into $\Omega\exterior$.

We define the space $\Hlocopspace[\Omega\exterior]{\Delta}{1}:=\{v\in \Hlocspace[\Omega\exterior]{1}:\Delta v\in\Ltwospace[\Omega\exterior]\}$, and
then we define the exterior Dirichlet and Neumann traces, $\trace[e]{0}:\Hlocspace[\Omega\exterior]{1}\to \Hspace{1/2}$
and $\trace[e]{1}:\Hlocopspace[\Omega\exterior]{\Delta}{1}\to \Hspace{-1/2}$, by
\begin{align}
\trace[e]{0} f(\x)&:=\lim_{\Omega\exterior\ni\y\to\x\in\Gamma}f(\y),\\
\trace[e]{1} f(\x)&:=\lim_{\Omega\exterior\ni\y\to\x\in\Gamma}\vec\nu_\x\cdot\nabla f(\y).
\end{align}
The interior traces $\trace[i]{0}$ and $\trace[i]{1}$ can be defined in the same way but taking the limit from within $\Omega\interior$.

We recall that if the Dirichlet and Neumann traces of a
solution of \cref{eq:Helmholtz} are known, then the potentials \cref{eq:single} and
\cref{eq:double} may be used to reconstruct the function in $\Omega\exterior$ using the
following relation.
\begin{equation}\label{eq:represent}
u = \popK(\trace[e]{0} u) - \popV(\trace[e]{1} u).
\end{equation}

It is also known \cite[lemma 6.6]{Stein07} that for all $\mu \in \Hspace{-1/2}$, the function
\begin{equation}\label{eq:helm_sing_rec}
u^{\popV}_\mu := \popV \mu
\end{equation}
satisfies $-\Delta u^{\popV}_\mu-k^2u^{\popV}_\mu = 0$ in $\mathbb{R}^3\setminus\Gamma$.
Similarly, for the double layer potential there holds \cite[lemma 6.10]{Stein07} that for all $v \in \Hspace{1/2}$, the function
\begin{equation}\label{eq:helm_doub_rec}
u^{\popK}_v := \popK v
\end{equation}
satisfies $-\Delta u^{\popK}_v-k^2u^{\popK}_v = 0$ in $\mathbb{R}^3\setminus\Gamma$.

We define
$\average{\trace{0} f}$ and $\average{\trace{1} f}$ 
to be the averages of the interior and exterior Dirichlet and Neumann traces of $f$.
We define the single layer, double layer, adjoint double layer, and hypersingular boundary integral operators,
$\bopV:\Hspace{-1/2}\to \Hspace{1/2}$,
$\bopK:\Hspace{1/2}\to \Hspace{1/2}$,
$\bopKadj:\Hspace{-1/2}\to \Hspace{-1/2}$, and
$\bopW:\Hspace{1/2}\to \Hspace{-1/2}$,
by
\begin{subequations}\label{eq:these_definitions}
\begin{align}
(\bopK v)(\x)&:=\average{\trace{0}\popK v}(\x),&
(\bopV \mu)(\x)&:=\average{\trace{0}\popV\mu}(\x),\\
(\bopW v)(\x)&:=-\average{\trace{1}\popK v}(\x),&
(\bopKadj\mu)(\x)&:=\average{\trace{1}\popV\mu}(\x),
\end{align}
\end{subequations}
where 
$\x\in\Gamma$,
$v\in \Hspace{1/2}$ and $\mu\in \Hspace{-1/2}$
\cite[chapter 6]{Stein07}.

We define $\jump{\trace{0}}:=\trace[e]{0}-\trace[i]{0}$ and $\jump{\trace{1}}:=\trace[e]{1}-\trace[i]{1}$
to be the jumps of the interior and exterior Dirichlet and Neumann traces across the boundary.
In \cite[chapter 6]{Stein07}, the following jump conditions are shown:
\begin{align}
\jump{\trace{0}}\popV=\jump{\trace{1}}\popK&=0,&
\jump{\trace{1}}\popV=
-\jump{\trace{0}}\popK&=-\bopI,
\label{eq:and_these_definitions}
\end{align}
where $\bopI$ is the identity operator.

It follows from \cref{eq:these_definitions,eq:and_these_definitions} that
\begin{subequations}
\begin{align}
\trace[e]{0}\popV &= \bopV,&
\trace[e]{1}\popV &= -\tfrac12\bopI+\bopKadj,\\
\trace[e]{0}\popK &= \tfrac12\bopI+\bopK,&
\trace[e]{1}\popK &= -\bopW,\\
\trace[i]{0}\popV &= \bopV,&
\trace[i]{1}\popV &= \tfrac12\bopI+\bopKadj,\label{crossref_this_trace}\\
\trace[i]{0}\popK &= -\tfrac12\bopI+\bopK,&
\trace[i]{1}\popK &= -\bopW.
\end{align}
\end{subequations}

We let $\Ltwoinner{\cdot}{\cdot}$ denote the $\Hspace{1/2}$--$\Hspace{-1/2}$ duality pairing. For all
$v,w\in\Hspace{1/2}$, $\mu,\eta\in\Hspace{-1/2}$, and $a\in\CC$,
this pairing satisfies
\begin{subequations}
\begin{align}
\Ltwoinner{v+w}{\mu}&=\Ltwoinner{v}{\mu}+\Ltwoinner{w}{\mu},&
\Ltwoinner{v}{\mu+\eta}&=\Ltwoinner{v}{\mu}+\Ltwoinner{v}{\eta},\\
\Ltwoinner{av}{\mu}&=a\Ltwoinner{v}{\mu},&
\Ltwoinner{v}{a\mu}&=\overline{a}\Ltwoinner{v}{\mu},
\end{align}
\end{subequations}
where $\overline{a}$ denotes the complex conjugate of $a$.

Following \cite[section 2.5]{Stein07}, the
norms $\Hnorm{1/2}{\cdot}$ and $\Hnorm{-1/2}{\cdot}$ are defined, for $v\in\Hspace{1/2}$ and $\mu\in\Hspace{-1/2}$,
by
\begin{align}
\Hnorm{1/2}{v}&
:=\left(\Ltwonorm{v}+\int_\Gamma\int_\Gamma\frac{\left(v(\vec{x})-v(\vec{y})\right)^2}{\seminorm{\vec{x}-\vec{y}}^3}\dx[\vec{x}]\dx[\vec{y}]\right)^{\tfrac12},\\
\Hnorm{-1/2}{\mu}&:=\sup_{w\in\Hspace{1/2}\setminus\{0\}}\frac{\seminorm{\Ltwoinner{w}{\mu}}}{\Hnorm{1/2}{w}}.
\end{align}

The following results are known for the single layer and hypersingular operators in
$\RR^3$.

\begin{lemma}[G{\aa}rding's inequality for $\bopV$]\label{lemma:Helmholtz_coercivity1}
There exists a compact operator $\bop{T}_\bop{V}:\Hspace{-1/2}\to\Hspace{1/2}$ and $\alpha_\bopV>0$ such that
\begin{align}
\alpha_\bopV \Hnorm{-1/2}{\mu}^2 &\leqslant \Ltwoinner{\bopV \mu}{\mu}+\Ltwoinner{\bop{T}_\bop{V}\mu}{\mu},
&&\forall \mu \in \Hspace{-1/2}.
\end{align}
\end{lemma}
\begin{proof}\cite[theorem 6.40]{Stein07}.\end{proof}

\begin{lemma}[G{\aa}rding's inequality for $\bopW$]\label{lemma:Helmholtz_coercivity2}
There exists a compact operator $\bop{T}_\bop{W}:\Hspace{1/2}\to\Hspace{-1/2}$ and $\alpha_\bopW>0$ such that
\begin{align}
\alpha_\bopW \Hnorm{1/2}{v}^2 &\leqslant \Ltwoinner{\bopW v}{v}+\Ltwoinner{\bop{T}_\bop{W}v}{v},
&&\forall v \in \Hspace{1/2}.
\end{align}
\end{lemma}
\begin{proof}
This follows by applying the proof of \cite[theorem 6.40]{Stein07} to the hypersingular operator.
\end{proof}

The following boundedness results are also known.

\begin{lemma}[Boundedness]\label{lemma:Helmholtz_boundedness}
There exist $C_\bopV,C_\bopK,C_{\bopKadj},C_\bopW>0$ such that
\begin{align}
\text{i)}&&\Hnorm{1/2}{\bopV \mu} &\leqslant C_\bopV \Hnorm{-1/2}{\mu} &&\forall \mu \in \Hspace{-1/2},\\
\text{ii)}&&\Hnorm{1/2}{\bopK v} &\leqslant C_\bopK \Hnorm{1/2}{v} &&\forall v \in \Hspace{1/2},\\
\text{iii)}&&\Hnorm{-1/2}{\bopKadj \mu} &\leqslant C_{\bopKadj} \Hnorm{-1/2}{\mu} &&\forall \mu \in \Hspace{-1/2},\label{K'bounded}\\
\text{iv)}&&\Hnorm{-1/2}{\bopW v} &\leqslant C_\bopW \Hnorm{1/2}{v} &&\forall v \in \Hspace{1/2}.
\end{align}
\end{lemma}
\begin{proof}
\cite[sections 6.2--6.5 and 6.9]{Stein07}.
\end{proof}

We define the exterior Calder\'on projector by
\begin{equation}\label{eq:calder}
\bop{C}\exterior:=
\begin{pmatrix}
\tfrac12\bopI+\bopK & -\bopV\\
-\bopW & \tfrac12\bopI-\bopKadj
\end{pmatrix},
\end{equation}
and recall that if $u$ is a solution of \cref{eq:Helmholtz_pm} then it satisfies
\begin{equation}\label{eq:calder_id}
\bop{C}\exterior\begin{pmatrix}\trace[e]{0} u\\\trace[e]{1} u\end{pmatrix}=\begin{pmatrix}\trace[e]{0} u\\\trace[e]{1} u\end{pmatrix}.
\end{equation}

When considering eigenvalues of the Laplacian, we will make use of the interior Calder\'on projector. This is defined by
\begin{equation}\label{eq:calder_interior}
\bop{C}\interior:=
\begin{pmatrix}
\tfrac12\bopI-\bopK & \bopV\\
\bopW & \tfrac12\bopI+\bopKadj
\end{pmatrix}.
\end{equation}
If $u\interior$ is a solution of an interior Helmholtz problem, then 
it satisfies
\begin{equation}\label{eq:calder_id_interior}
\bop{C}\interior\begin{pmatrix}\trace[i]{0} u\interior\\\trace[i]{1} u\interior\end{pmatrix}=\begin{pmatrix}\trace[i]{0} u\interior\\\trace[i]{1} u\interior\end{pmatrix}.
\end{equation}

Taking the product of the exterior identity \cref{eq:calder_id} with two test functions,
we arrive at the following equations.
\begin{align}
\Ltwoinner{\trace[e]{0} u}{\mu} &= \Ltwoinner{(\tfrac12\bopI + \bopK)\trace[e]{0} u}{\mu}
    -\Ltwoinner{\bopV\trace[e]{1}u}{\mu} &&\forall \mu \in \Hspace{-1/2},\label{eq:Calderon_1}\\
\Ltwoinner{\trace[e]{1} u}{v} &= \Ltwoinner{(\tfrac12\bopI - \bopKadj)\trace[e]{1} u}{v} - \Ltwoinner{\bopW \trace[e]{0}u}{v} &&\forall v \in \Hspace{1/2}.\label{eq:Calderon_2}
\end{align}

For a more compact notation, we write $u$ in the place of $\trace[e]{0} u$ and
introduce $\lambda=\trace[e]{1} u$ and the exterior Calder\'on form
\begin{multline}\label{eq:compact_form}
\form{C}\exterior[(u,\lambda),(v,\mu)]:=
\Ltwoinner{(\tfrac12\bopI+\bopK)u}{\mu} - \Ltwoinner{\bopV\lambda}{\mu}\\
+
\Ltwoinner{(\tfrac12\bopI - \bopKadj)\lambda}{v} - \Ltwoinner{\bopW u}{v}.
\end{multline}
We may then rewrite \cref{eq:Calderon_1,eq:Calderon_2} as
\begin{equation}\label{eq:realtion}
\form{C}\exterior[(u,\lambda),(v,\mu)]=\Ltwoinner{u}{\mu}+\Ltwoinner{\lambda}{v}.
\end{equation}

We will also frequently use the multitrace form, defined by
\begin{equation}\label{eq:mult_trace}
\form{A}[(u,\lambda),(v,\mu)]:=
-\Ltwoinner{\bopK u}{\mu} + \Ltwoinner{\bopV\lambda}{\mu}
+
\Ltwoinner{\bopKadj\lambda}{v} + \Ltwoinner{\bopW u}{v},
\end{equation}
and the multitrace form with compact perturbation,
\begin{multline}\label{eq:mult_trace_pert}
\form{A}_\bop{T}[(u,\lambda),(v,\mu)]:=
-\Ltwoinner{\bopK u}{\mu} + \Ltwoinner{(\bopV+\bop{T}_\bopV)\lambda}{\mu}
\\
+
\Ltwoinner{\bopKadj\lambda}{v} + \Ltwoinner{(\bopW+\bop{T}_\bopW) u}{v},
\end{multline}
where $\bop{T}_\bopV$ and $\bop{T}_\bopW$ are the compact operators from \cref{lemma:Helmholtz_coercivity1,lemma:Helmholtz_coercivity2}.
Using \cref{eq:mult_trace}, we may rewrite \cref{eq:realtion} as 
\begin{equation}\label{eq:skewsym_relation}
\form{A}[(u,\lambda),(v,\mu)]=-\tfrac12\Ltwoinner{u}{\mu}-\tfrac12\Ltwoinner{\lambda}{v}.
\end{equation}

To quantify the two traces we introduce the product space 
\[
\productspace{V} := \Hspace{1/2} \times \Hspace{-1/2}.
\]
We also introduce the associated norm
\[
\Vnorm{(v,\mu)}:= \Hnorm{1/2}{v}+\Hnorm{-1/2}{\mu}.
\]

Using the results in \cref{lemma:Helmholtz_coercivity1,lemma:Helmholtz_coercivity2,lemma:Helmholtz_boundedness},
we obtain the continuity and coercivity of $\form{A}$.
\begin{lemma}[Continuity]\label{lemma:cont_cald}
There exists $C>0$ such that
\begin{align*}
\left|\form{A}[(w,\eta),(v,\mu)]\right| &\leqslant C \Vnorm{(w,\eta)}\Vnorm{(v,\mu)}&&\forall(w,\eta),(v,\mu)\in\productspace{V}.
\end{align*}
\end{lemma}
\begin{proof}
Use \cref{lemma:Helmholtz_boundedness}.
\end{proof}

\begin{lemma}[Coercivity]\label{lemma:coerciv_cald}
There exists $\alpha>0$ and compact operators $\bop{T}_\bop{V}:\Hspace{-1/2}\to\Hspace{1/2}$ and $\bop{T}_\bop{W}:\Hspace{1/2}\to\Hspace{-1/2}$ such that
\begin{multline*}
\alpha\left(
\Hnorm{1/2}{v}^2+\Hnorm{-1/2}{\mu}^2
\right)
\leqslant
\form{A}[(v,\mu),(v,\mu)]
+\Ltwoinner{\bop{T}_\bop{W}v}{v}
+\Ltwoinner{\bop{T}_\bop{V}\mu}{\mu}
\\
\forall(v,\mu)\in\productspace{V}.
\end{multline*}
\end{lemma}
\begin{proof}
Use the coercivity of $\bopV$ and $\bopW$ from \cref{lemma:Helmholtz_coercivity1,lemma:Helmholtz_coercivity2} and let $\alpha=\min(\alpha_\bopW,\alpha_\bopV)$.
\end{proof}

\section{Derivation of a formulation for Helmholtz Dirichlet problems}\label{sec:derive}
In this section, we derive a formulation for the exterior Helmholtz problem with
non-homogeneous Dirichlet conditions.

As in \cite{BeBuScro18}, we write the boundary condition as
\begin{equation}
R_\Gamma(u,\lambda) = 0,
\end{equation}
and look to solve
\begin{equation}
\form{A}[(u,\lambda),(v,\mu)]=-\tfrac12\Ltwoinner{u}{\mu} - \tfrac12\Ltwoinner{\lambda}{v}+\Ltwoinner{R_\Gamma(u,\lambda)}{\overline{\betaparam_1} v + \overline{\betaparam_2} \mu},\label{eq:multi_bc}
\end{equation}
for some $\betaparam_1,\betaparam_2\in\mathbb{C}$.

\subsection{Dirichlet boundary condition}\label{sec:helmdirichlet}
To impose a Dirichlet boundary condition,
we choose $\betaparam_1 = -\ii\betaparam\D^{1/2}$, $\betaparam_2 = \ii\betaparam\D^{-1/2}$, where $\betaparam\D$ will
be identified with a mesh-dependent penalty parameter, and
\begin{equation}\label{helmeq:Dir_R}
R\D(u,\lambda) := \ii\betaparam\D^{1/2} (g\D-u)
\end{equation}
where $g\D \in \Hspace{1/2}$ is the Dirichlet data.

Inserting this into \cref{eq:multi_bc}, we obtain the formulation
\begin{equation}\label{helmeq:multi_D}
\form{A}[(u,\lambda),(v,\mu)]-\tfrac12\Ltwoinner{u}{\mu} +\tfrac12\Ltwoinner{\lambda}{v}
+\Ltwoinner{\betaparam\D u}{v}=\Ltwoinner{g\D}{\overline{\betaparam\D} v-\mu}.
\end{equation}

This leads us to the following formulation for the Helmholtz Dirichlet problem:
Find
$(u,\lambda) \in \productspace{V}$ such that
\begin{align}\label{eq:abstract_form_helmDir}
\form{A}[(u,\lambda),(v,\mu)]+\form{B}\exterior\D[(u,\lambda),(v,\mu)] &=
\form{L}\exterior\D(v,\mu)&&\forall(v,\mu) \in \productspace{V},
\end{align}
where
\begin{align}\label{eq:operator_helmDirichlet}
\form{B}\exterior\D[(u,\lambda),(v,\mu)]&:=\tfrac12\Ltwoinner{\lambda}{v}-\tfrac12\Ltwoinner{u}{\mu}+\Ltwoinner{\betaparam\D u}{v},\\ 
\label{eq:helmLDir} 
\form{L}\exterior\D(v,\mu) &:=\Ltwoinner{g\D}{\overline{\betaparam\D} v - \mu}.
\end{align}

We now show that a solution of the Helmholtz problem \cref{eq:Helmholtz} is also a solution of this weak formulation.
\begin{proposition}
If $u$ is a solution of \cref{eq:Helmholtz}, then $(\trace[e]{0}u,\trace[e]{1}u)$ is a solution of \cref{eq:abstract_form_helmDir}.
\end{proposition}
\begin{proof}
Let $(v, \mu)\in\productspace{V}$.
By \cref{eq:skewsym_relation}, we see that
\begin{align*}
(\form{A}+\form{B}\D\exterior)[(\trace[e]{0}u,\trace[e]{1}u),(v,\mu)]
\hspace{-3.5cm}&\\
&= -\tfrac12\Ltwoinner{\trace[e]{0}u}{\mu}-\tfrac12\Ltwoinner{\trace[e]{1}u}{v}
+\tfrac12\Ltwoinner{\trace[e]{1}u}{v}-\tfrac12\Ltwoinner{\trace[e]{0}u}{\mu}+\Ltwoinner{\betaparam\D\trace[e]{0}u}{v},
\\
&= -\Ltwoinner{\trace[e]{0}u}{\mu}+\Ltwoinner{\betaparam\D\trace[e]{0}u}{v}\\ 
&= \Ltwoinner{\trace[e]{0}u}{\overline{\betaparam\D} v-\mu}.
\end{align*}
Using \cref{eq:Laplace_diribc}, we see that 
\begin{align*}
(\form{A}+\form{B}\D\exterior)[(\trace[e]{0}u,\trace[e]{1}u),(v,\mu)]
&= \Ltwoinner{g\D}{\overline{\betaparam\D} v-\mu}\\
&= \form{L}\exterior\D(v,\mu).
\end{align*}
\end{proof}

To discretise \cref{eq:abstract_form_helmDir}, we introduce a family of conforming, shape regular triangulations of
$\Gamma$, $\{\Th\}_{h>0}$, indexed by the largest element
diameter of the mesh, $h$.
We then consider the following finite element spaces.
\begin{align*}
\Poly^p_h &:= \{v_h \in C^0(\Gamma): v_h\vert_{T} \in \mathbb{P}_p(T) \text{, for every }T\in\Th\},\\
\DPoly^q_h &:= \{v_h \in \Ltwo(\Gamma): v_h\vert_{T} \in \mathbb{P}_q(T) \text{, for every }T\in\Th\},
\end{align*}
where $\mathbb{P}_p(T)$ denotes the space of polynomials of order less
than or equal to $p$ on a triangle $T$,
and $\{\Gamma_i\}_{i=1}^M$ are the polygonal faces of $\Gamma$.
We observe that $\Poly^p_h \subset \Hspace{1/2}$ and $\DPoly^q_h\subset \Ltwo(\Gamma)$.

We let $\productspace{V}_h$ be a discrete product space: in our case, we take this to be equal to either
$\Poly_h^p(\Gamma)\times\DPoly^q_h(\Gamma)$ or $\Poly_h^p(\Gamma)\times\Poly^q_h(\Gamma)$, although a wide range
of other choices are possible.

Using the space $\productspace{V}_h$, we look to solve the discrete problem:
Find $(u_h,\lambda_h) \in \productspace{V}_h$ such that
\begin{align}\label{eq:discrete}
\form{A}[(u_h,\lambda_h),(v_h,\mu_h)]+\form{B}\D\exterior[(u_h,\lambda_h),(v_h,\mu_h)] &=
\form{L}\D\exterior(v_h,\mu_h) &&\forall (v_h,\mu_h) \in \productspace{V}_h.
\end{align}

We define the norm
\[
\Bnorm[\D]{(v,\mu)} := \Vnorm{(v,\mu)} + \seminorm{\betaparam\D}^{1/2} \Ltwonorm[\Gamma]{v}.
\]

We note that since $\betaparam\D$ may be dependent on $h$, this norm is not equivalent to the norm $\Vnorm{\cdot}$
independently of $h$.

\section{Analysis}\label{sec:analysis}
In this section, we analyse the formulation derived in the previous section.
Throughout the analysis, we will use the following notation.

\begin{definition}
For two quantities $a$ and $b$ that may vary with $h$, we write $a\lesssim b$ if
there is a constant $h_0>0$ and an $h$-independent constant $c\in\RR$
such that $a\leqslant cb$ for all $h<h_0$.
We write $a\eqsim b$ if $a\lesssim b$ and $b\lesssim a$.
\end{definition}

\subsection{Analysis of the continuous problem}
We begin by analysing the continuous problem \cref{eq:abstract_form_helmDir}.
In the same way as we did in \cite{BeBuScro18}, we prove that the form $\form{A}+\form{B}\exterior\D$ is continuous.
\begin{proposition}[Continuity]\label{diri:continuity}
There exists
$M>0$ such that $\forall(w,\eta),(v,\mu) \in \productspace{V}$,
\[\seminorm{\form{A}[(w,\eta),(v,\mu)] +
\form{B}\exterior\D[(w,\eta),(v,\mu)]} \leqslant M\Bnorm[\D]{(w,\eta)}\Bnorm[\D]{(v,\mu)}.
\]
\end{proposition}
\begin{proof}
This can be proved in the same way as \cite[proposition 4.13]{BeBuScro18}, but with $\seminorm{\betaparam\D}$
in the place of $\betaparam\D$.
\end{proof}

Eigenvalues of the interior Laplacian have an effect on the boundary integral formulation of the exterior problem.
These eigenvalues are defined as follows.

\begin{definition}[Eigenvalues of the Laplacian]
If the interior Laplace problem: Find $u$ such that
\begin{subequations}\label{laplace_diri}
\begin{align}
-\Delta u &= \eigenval\D u &&\text{in }\Omega\interior,\label{l__diri1}\\
u &=0 &&\text{on }\Gamma,
\end{align}
\end{subequations}
has multiple solutions, then $\eigenval\D$ is called
a Dirichlet eigenvalue of the interior Laplacian.

If the interior Laplace problem: Find $u$ such that
\begin{subequations}\label{laplace_neu}
\begin{align}
-\Delta u &= \eigenval\N u &&\text{in }\Omega\interior,\\
\frac{\partial u}{\partial\vec\nu} &=0 &&\text{on }\Gamma,
\end{align}
\end{subequations}
has multiple solutions that differ by more than a constant, then $\eigenval\N$ is called
a Neumann eigenvalue of the interior Laplacian.

If the interior Laplace problem: Find $u$ such that
\begin{subequations}\label{laplace_robin}
\begin{align}
-\Delta u &= \eigenval\R u &&\text{in }\Omega\interior,\\
\frac{\partial u}{\partial\vec\nu} + \betaparam\D u &=0 &&\text{on }\Gamma,
\end{align}
\end{subequations}
has multiple solutions, then we call $\eigenval\R$
a Robin eigenvalue of the interior Laplacian with Robin parameter $\betaparam\D$.
\end{definition}

We now prove some important properties of these eigenvalues.
\begin{lemma}\label{lemma:not_both}
For $l\in\CC$, at most one of the following statements is true:
\begin{itemize}
\item $l$ is a Dirichlet eigenvalue of the interior Laplacian;
\item $l$ is a Neumann eigenvalue of the interior Laplacian;
\item $l$ is a Robin eigenvalue of the interior Laplacian (for some $\betaparam\D\not=0$).
\end{itemize}
\end{lemma}
\begin{proof}
By \cite[section 7.6]{Stein07}, we know that a value cannot be both a Dirichlet and a Neumann eigenvalue.

If $l$ is both a Dirichlet and a Robin eigenvalue, then there exist $u_1$ and $u_2$ satisfying
\cref{l__diri1}
whose values on $\Gamma$ satisfy $u_1=u_2=0$ and
\begin{align*}
0
&=\frac{\partial u_1}{\partial\vec\nu} + \betaparam\D u_1
=\frac{\partial u_1}{\partial\vec\nu},\\
0
&=\frac{\partial u_2}{\partial\vec\nu} + \betaparam\D u_2
=\frac{\partial u_2}{\partial\vec\nu}.
\end{align*}
If $u_1$ and $u_2$ differed by a constant, their values on $\Gamma$ would differ by the same constant,
contradicting the boundary conditions $u_1=u_2=0$.
Therefore $l$ is also a Neumann eigenvalue, which is a contradiction.

Similarly, if $l$ is both a Neumann and a Robin eigenvalue, then in the same way we see that $l$ is also a Dirichlet
eigenvalue, leading to a similar contradiction.
\end{proof}

\begin{lemma}\label{lemma:robin_nontrivial}
If $\Im(\betaparam\D)\not=0$, then the interior Laplacian with Robin parameter $\betaparam\D$
has no non-trivial real Robin eigenvalues.
\end{lemma}
\begin{proof}
Let $\betaparam\D\in\CC$ with $\Im(\betaparam\D)\not=0$.
Suppose that $\eigenval\R\in\RR\setminus\{0\}$ is a Robin eigenvalue of the
interior Laplacian with corresponding eigenfunction $u\R$, ie
\begin{align*}
-\Delta u\R &= \eigenval\R u\R &&\text{in }\Omega\interior,\\
\frac{\partial u\R}{\partial\vec\nu} + \betaparam\D u\R &=0 &&\text{on }\Gamma.
\end{align*}
Consider the weak formulation of this problem: Find $u\R\in\Hspace[\Omega]{1}$ such that
\begin{align*}
\Ltwoinner[\Omega]{\nabla u\R}{\nabla v} + \betaparam\D\Ltwoinner{u\R}{v}&=\eigenval\R\Ltwoinner[\Omega]{u\R}{v},
&&\forall v\in\Hspace[\Omega]{1}.
\end{align*}
Taking $v=u\R$ leads to
\begin{align*}
\Ltwonorm[\Omega]{\nabla u\R}^2 + \betaparam\D\Ltwonorm{u\R}^2&=\eigenval\R\Ltwoinner[\Omega]{u\R}{u\R}^2.
\end{align*}
Taking the imaginary part of this gives
\begin{align*}
\Im(\betaparam\D)\Ltwonorm{u\R}^2&=0.
\end{align*}
As $\Im(\betaparam\D)\not=0$, this implies that
$\Ltwonorm{u\R}=0$,
and so $u\R=0$ on $\Gamma$.

This, however, implies that $u\R$ is also a Dirichlet eigenfunction of the Laplacian with eigenvalue
$\eigenval\R$, contradicting \cref{lemma:not_both}. Hence no such real eigenvalue exists.
\end{proof}

We now proceed to prove that the form $\form{A} + \form{B}\exterior\D$ is injective. First we prove this
when $k^2$ is not a Dirichlet eigenvalue. Note that by \cref{lemma:robin_nontrivial}, the assumption that
$k^2$ is not a Robin eigenvalue of the interior Laplacian holds whenever $\Im(\betaparam\D)\not=0$.

\begin{lemma}[Injectivity, part one]\label{diri:injective1}
Let $(v,\mu)\in\productspace{V}$.
If
$k^2$ is not a Dirichlet eigenvalue of the interior Laplacian and
$k^2$ is not a Robin eigenvalue of the interior Laplacian with Robin parameter $\betaparam\D$, 
then $\forall(w,\eta)\in\productspace{V}$
\begin{equation*}
\form{A}[(v,\mu),(w,\eta)] + \form{B}\exterior\D[(v,\mu),(w,\eta)] = 0
\end{equation*}
implies that $(v,\mu)=0$.
\end{lemma}
\begin{proof}
Suppose that $(v,\mu)\in\productspace{V}$ such that
\begin{align}
(\form{A}+\form{B}\exterior\D)[(v,\mu),(w,\eta)]&=0
&&\forall(w,\eta)\in\productspace{V}.\label{injeq:1}
\end{align}

Taking $w=0$ in \cref{injeq:1}, we see that, for all $\eta\in\Hspace{-1/2}$,
\begin{align}
\Ltwoinner{\bopV\mu}{\eta}
&=
\Ltwoinner{(\tfrac12\bopI+\bopK) v}{\eta}.\label{injeq:2}
\end{align}

$k^2$ is not a Dirichlet eigenvalue of the interior Laplacian, so by \cite[section 7.6]{Stein07}
we see that $\bopV$ is invertible, and (as \cref{injeq:2} is a direct boundary integral formulation)
there exists a solution to the interior Helmholtz equation $\tilde{p}\in\Hspace[\Omega\interior]{1}$ such that
\begin{align*}
\trace[i]{0}\tilde{p}&=v,\\
\trace[i]{1}\tilde{p}&=\mu.
\end{align*}

Taking $\eta=0$ in \cref{injeq:1}, we see that, for all $w\in\Hspace{1/2}$,
\begin{align}
0&=
\Ltwoinner{\bopW v}{w}+\Ltwoinner{\betaparam\D v}{w} + \Ltwoinner{(\tfrac12\bopI+\bopKadj) \mu}{w}.\label{injeq:3}
\end{align}
We also know from the second line of the interior Calder\'on identity \cref{eq:calder_id_interior} that
\begin{align}
\Ltwoinner{\bopW v}{w} + \Ltwoinner{(\tfrac12\bopI+\bopKadj) \mu}{w}
&=
\Ltwoinner{\bopW \trace[i]{0}\tilde{p}}{w} + \Ltwoinner{(\tfrac12\bopI+\bopKadj) \trace[i]{1}\tilde{p}}{w}
\\&= \Ltwoinner{\trace[i]{1}\tilde{p}}{w}\\&=\Ltwoinner{\mu}{w},
\end{align}
and so using \cref{injeq:3} we see that
\begin{align}
0&=
\Ltwoinner{\mu}{w}+\Ltwoinner{\betaparam\D v}{w},
\end{align}
and so $\mu=-\betaparam\D v$.

This means that $\tilde{p}$ is the solution of \cref{laplace_robin} with $\eigenval\R=k^2$.
Since $k^2$ is not a Robin eigenvalue, the unique solution of \cref{laplace_robin} is $\tilde{p}=0$,
and so $v=\mu=0$.
\end{proof}

In order to prove injectivity at Dirichlet eigenvalues, we require the following lemma.

\begin{lemma}\label{lemma:surjective}
If $k^2$ is not a Neumann eigenvalue of the interior Laplacian, then the operator $\tfrac12\bopI+\bopKadj$ is surjective.
\end{lemma}
\begin{proof}
We prove this lemma using \cite[theorem 2.20]{brezis-book}. This states that if a linear operator $\bop{F}:\Hspace{-1/2}\to\Hspace{-1/2}$
is densely defined and closed, then $\bop{F}$ is surjective if and only if
\begin{align*}\Hnorm{1/2}{v}&\leqslant c\Hnorm{1/2}{\bop{F}'v}&&\forall v\in\Hspace{1/2}.\end{align*}

Let $\bop{F}=\tfrac12\bopI+\bopK'$.
$\bop{F}$ is (trivially) densely defined. Using \cref{K'bounded}, we see that for all $v\in\Hspace{1/2}$
\begin{align*}
\Hnorm{1/2}{\bop{F}v}
&\leqslant \Hnorm{1/2}{\tfrac12v}+\Hnorm{1/2}{\bopKadj v}
\\
&\leqslant (\tfrac12+C_{\bopKadj})\Hnorm{1/2}{v},
\end{align*}
and so $\bop{F}$ is bounded. This implies that $\bop{F}$ is closed.

Taking the adjoint, we see that $\bop{F}'=\tfrac12\bopI+\bopK$. In \cite[section 7.6]{Stein07}, it is shown that if $k^2$ is not a Neumann eigenvalue of the interior Laplacian, then
the equation \[\bop{F}'v=g\] has a unique solution for any $g\in\Hspace{1/2}$. This implies that the operator $\left(\bop{F}'\right)^{-1}$ is well-defined.
As $\bop{F}$ is bounded, it follows from \cite[corollary 2.7]{brezis-book} that $\left(\bop{F}'\right)^{-1}$ is also bounded. Hence,
\begin{align*}
\Hnorm{1/2}{v}
&=\Hnorm{1/2}{\left(\bop{F}'\right)^{-1}g}\\
&\leqslant c\Hnorm{1/2}{g}\\
&=
c\Hnorm{1/2}{\bop{F}'v}.
\end{align*}

We therefore conclude by \cite[theorem 2.20]{brezis-book} that $\bop{F}$ is surjective.
\end{proof}

We can now prove that $\form{A} + \form{B}\exterior\D$ is injective when $k^2$ is a Dirichlet eigenvalue.
\begin{lemma}[Injectivity, part two]\label{diri:injective2}
Let $(v,\mu)\in\productspace{V}$.
If $k^2$ is a Dirichlet eigenvalue of the interior Laplacian, 
then $\forall(w,\eta)\in\productspace{V}$
\begin{equation*}
\form{A}[(v,\mu),(w,\eta)] + \form{B}\exterior\D[(v,\mu),(w,\eta)] = 0
\end{equation*}
implies that $(v,\mu)=0$.
\end{lemma}
\begin{proof}
Suppose that $(v,\mu)\in\productspace{V}$ such that
\begin{align}
(\form{A}+\form{B}\exterior\D)[(v,\mu),(w,\eta)]&=0
&&\forall(w,\eta)\in\productspace{V}.\label{inj2eq:1}
\end{align}

Taking $w=0$ in \cref{inj2eq:1}, we see that, for all $\eta\in\Hspace{-1/2}$,
\begin{align}
\Ltwoinner{\bopV\mu}{\eta}
&=
\Ltwoinner{(\tfrac12\bopI+\bopK) v}{\eta}.\label{inj2eq:2}
\end{align}

As $k^2$ is a Dirichlet eigenvalue of the interior Laplacian, then we see by \cref{lemma:not_both} that
$k^2$ is not a Neumann eigenvalue or a Robin eigenvalue.

As $k^2$ is not a Neumann eigenvalue, we see by \cite[section 7.6]{Stein07} that the equation
\begin{align*}
\Ltwoinner{(\tfrac12\bopI+\bopK)u}{\eta}&=\Ltwoinner{g}{\eta}&&\forall\eta\in\Hspace{-1/2}
\end{align*}
has a unique solution. By \cref{inj2eq:2}, we see that $v$ is the solution of this equation with $g=\bopV\mu$.

Let $\tilde{p}=\popV\mu-\popK v$. By \cref{eq:helm_sing_rec,eq:helm_doub_rec}, we see that
\begin{align}
-\Delta \tilde{p} - k^2\tilde{p} &=0&&\text{in }\Omega\interior.\label{inj2eq:8}
\end{align}
Taking the Dirichlet trace of $\tilde{p}$ and applying \cref{inj2eq:2}, we see that on $\Gamma$,
\begin{align}
\trace[i]{0}\tilde{p}
&=\bopV\mu + (\tfrac12\bopI-\bopK)v
\notag\\
&=(\tfrac12\bopI+\bopK)v + (\tfrac12\bopI-\bopK)v
\notag\\
&=v.\label{inj2eq:6}
\end{align}
We therefore conclude, using the first line of the interior Calder\'on identity \cref{eq:calder_id_interior} and \cref{inj2eq:2}, that
\begin{align}
\bopV\trace[i]{1}\tilde{p}
&=
(\tfrac12\bopI+\bopK)\trace[i]{0}\tilde{p}
\notag\\
&=
(\tfrac12\bopI+\bopK)v
\notag\\
&=
\bopV \mu.\label{inj2eq:10}
\end{align}
Note that as $k^2$ is a Dirichlet eigenvalue, $\bopV$ is not injective so this does not necessarily imply that
$\trace[i]{1}\tilde{p}=\mu$.

As $k^2$ is not a Neumann eigenvalue, we can apply \cref{lemma:surjective} to show that 
the boundary integral equation
\begin{align}\label{inj2eq:16}
(\tfrac12\bopI+\bopKadj)f &= \mu-\trace[i]{1}\tilde{p}.
\end{align}
has at least one solution $f\in\Hspace{-1/2}$.
Let $\tilde{q}=\popV f$. By \cref{eq:helm_sing_rec}, we see that
\begin{align}
-\Delta \tilde{q} - k^2\tilde{q} &=0&&\text{in }\Omega\interior.\label{inj2eq:9}
\end{align}
Taking the Neumann trace of $\tilde{q}$ and using \cref{crossref_this_trace} and \cref{inj2eq:16}, we see that on $\Gamma$,
\begin{align}
\trace[i]{1}\tilde{q}
&=\trace[i]{1}(\popV f)
\notag\\
&=(\tfrac12\bopI+\bopKadj)f
\notag\\
&=\mu-\trace[i]{1}\tilde{p}.\label{inj2eq:14}
\end{align}
We know from the first line of the interior Calder\'on identity \cref{eq:calder_id_interior} that
\begin{align*}
(\tfrac12\bopI+\bopK)\trace[i]{0}\tilde{q}
&=
\bopV\trace[i]{1}\tilde{q}.
\end{align*}
Applying \cref{inj2eq:14} then using \cref{inj2eq:10} gives
\begin{align}
(\tfrac12\bopI+\bopK)\trace[i]{0}\tilde{q}
&=
\bopV(\mu-\trace[i]{1}\tilde{p})
\notag\\
&=
0.\label{inj2eq:12}
\end{align}
By \cite[lemma 7.6]{Stein07}, $\tfrac12\bopI+\bopK$ is injective, so $\trace[i]{0}\tilde{q}=0$.
By the interior Calder\'on identities \cref{eq:calder_id_interior}, this implies that $(\tfrac12\bopI-\bopKadj)\trace[i]{1}\tilde{q}=0$,
and so by \cref{inj2eq:14},
\begin{align}\label{inj2eq:4}
(\tfrac12\bopI-\bopKadj)\trace[i]{1}\tilde{p}& =(\tfrac12\bopI-\bopKadj)\mu.
\end{align}

Using \cref{inj2eq:4} and the second line of the interior Calder\'on identity \cref{eq:calder_id_interior}, we see that, for all $w\in\Hspace{1/2}$,
\begin{align}
\Ltwoinner{\bopW v}{w} + \Ltwoinner{(\tfrac12\bopI+\bopKadj) \mu}{w}
&=
\Ltwoinner{\bopW v}{w}
+ \Ltwoinner{\mu}{w}
- \Ltwoinner{(\tfrac12\bopI-\bopKadj) \mu}{w}
\notag\\
&=
\Ltwoinner{\bopW \trace[i]{0}\tilde{p}}{w}
+ \Ltwoinner{\mu}{w}
- \Ltwoinner{(\tfrac12\bopI-\bopKadj) \trace[i]{1}\tilde{p}}{w}
\notag\\&=
\Ltwoinner{\mu}{w}.\label{inj2eq:5}
\end{align}
Taking $\eta=0$ in \cref{inj2eq:1} and applying \cref{inj2eq:5}, we see that
\begin{align}
0&=
\Ltwoinner{\bopW v}{w}+\Ltwoinner{\betaparam\D v}{w} + \Ltwoinner{(\tfrac12\bopI+\bopKadj) \mu}{w}
\notag\\&=
\Ltwoinner{\mu}{w}+\Ltwoinner{\betaparam\D v}{w},
\label{inj2eq:3}
\end{align}
and so $\mu=-\betaparam\D v$.

Consider $\tilde{p}+\tilde{q}$. From \cref{inj2eq:8,inj2eq:9}, we see that
\begin{align}
-\Delta (\tilde{p}+\tilde{q}) - k^2(\tilde{p}+\tilde{q}) &=0&&\text{in }\Omega\interior.
\end{align}
Using the fact that $\trace[i]{0}\tilde{q}=0$ (as we concluded from \cref{inj2eq:12})
and \cref{inj2eq:6,inj2eq:14}, we see that on $\Gamma$,
\begin{align}
\trace[i]{0}(\tilde{p}+\tilde{q}) &=
\trace[i]{0}\tilde{p}+\trace[i]{0}\tilde{q}
\notag\\
&=v,\label{inj2eq:13}\\
\trace[i]{1}(\tilde{p}+\tilde{q}) &=
\trace[i]{1}\tilde{p}+\trace[i]{1}\tilde{q}
\notag\\
&=\mu,\label{inj2eq:11}
\end{align}
and so, using \cref{inj2eq:3}, we see that
\begin{align}
\frac{\partial (\tilde{p}+\tilde{q})}{\partial\vec\nu} + \betaparam\D (\tilde{p}+\tilde{q})
&= \mu + \betaparam\D v
\notag\\
&=0.
\end{align}
Therefore $\tilde{p}+\tilde{q}$ is solution of the Robin problem \cref{laplace_robin}. Since
$k^2$ is not a Robin eigenvalue, we see that $\tilde{p}+\tilde{q}=0$ in $\Omega\interior$,
and so by \cref{inj2eq:13,inj2eq:11} $v=\mu=0$.
\end{proof}

We now combine the previous two results.

\begin{proposition}[Injectivity]\label{diri:injective}
Let $(v,\mu)\in\productspace{V}$.
If
$k^2$ is not a Robin eigenvalue of the interior Laplacian with Robin parameter $\betaparam\D$, 
then $\forall(w,\eta)\in\productspace{V}$
\begin{equation*}
\form{A}[(v,\mu),(w,\eta)] + \form{B}\exterior\D[(v,\mu),(w,\eta)] = 0
\end{equation*}
implies that $(v,\mu)=0$.
\end{proposition}
\begin{proof}
Combine \cref{diri:injective1,diri:injective2}.
\end{proof}

Using \cref{lemma:robin_nontrivial}, we see that the following corollary holds.
\begin{corollary}
If $\Im(\betaparam\D)\not=0$ and $\forall(w,\eta)\in\productspace{V}$
\begin{equation*}
\form{A}[(v,\mu),(w,\eta)] + \form{B}\exterior\D[(v,\mu),(w,\eta)] = 0,
\end{equation*}
then $(v,\mu)=0$.
\end{corollary}

We now prove that the form $\form{A}_\bop{T}+\form{B}\exterior\D$ is coercive, where $\form{A}_\bop{T}$ is
the multitrace form with compact perturbation as defined in \cref{eq:mult_trace_pert}.

\begin{lemma}[coercivity]\label{lemma:AT_coercive}
If $\Re(\betaparam\D)\geqslant0$ then there exists $\alpha>0$ such that for all $(v,\mu)\in\productspace{V}$,
\[
\alpha\Vnorm{(v,\mu)}^2\leqslant\seminorm{(\form{A}_\bop{T}+\form{B}\exterior\D)[(v,\mu),(v,\mu)]}.
\]
\end{lemma}
\begin{proof}
Using the definitions of $\form{A}_\bop{T}$ and $\form{B}\exterior\D$, we see that
\begin{align*}
\Re((\form{A}_\bop{T}+\form{B}\exterior\D)[(v,\mu),(v,\mu)])
\hspace{-2cm}&\\
&=
\Re(
-\Ltwoinner{\bopK v}{\mu} + \Ltwoinner{(\bopV+\bop{T}_\bopV)\mu}{\mu}
+\Ltwoinner{\bopKadj\mu}{v}
\\&\hspace{1cm}
+ \Ltwoinner{(\bopW+\bop{T}_\bopW) v}{v}
+\tfrac12\Ltwoinner{\mu}{v}-\tfrac12\Ltwoinner{v}{\mu}+\Ltwoinner{\betaparam\D v}{v}
)
\\
&=
\Ltwoinner{(\bopV+\bop{T}_\bopV)\mu}{\mu}
+ \Ltwoinner{(\bopW+\bop{T}_\bopW) v}{v}
+\Ltwoinner{\Re(\betaparam\D) v}{v}.
\end{align*}
Applying \cref{lemma:Helmholtz_coercivity1,lemma:Helmholtz_coercivity2} and using the assumption
that $\Re(\betaparam\D)\geqslant0$, we see that
\begin{align*}
\Re((\form{A}_\bop{T}+\form{B}\exterior\D)[(v,\mu),(v,\mu)])
&\geqslant
\alpha_\bopV \Hnorm{-1/2}{\mu}^2
+\alpha_\bopW \Hnorm{1/2}{v}^2
+\Re(\betaparam\D)\Ltwonorm{v}^2
\\
&\geqslant
\alpha_\bopV \Hnorm{-1/2}{\mu}^2
+\alpha_\bopW \Hnorm{1/2}{v}^2
\\
&\geqslant
\alpha \Vnorm{(v,\mu)}^2,
\end{align*}
for some $\alpha>0$.

The result then follows from that fact that for any $c\in\CC$, $\seminorm{c}\geqslant\Re(c)$.
\end{proof}

We conclude our analysis of the continuous problem by showing that it is well-posed.

\begin{proposition}\label{lemma:c_unique}
If $\Re(\betaparam\D)>0$ and
$k^2$ is not a Robin eigenvalue of the interior Laplacian with Robin parameter $\betaparam\D$, 
then
the continuous problem \cref{eq:abstract_form_helmDir} has a unique solution.
\end{proposition}
\begin{proof}
As shown in \cite[theorem 3.15]{Stein07}, this is a direct consequence of \cref{diri:injective,lemma:AT_coercive}.
\end{proof}

\subsection{Analysis of the discrete problem}

We can now use \cite[theorem 4.2.9]{sauter-schwab} to show that the discrete problem \cref{eq:discrete} is well-posed,
and prove a quasi-optimal error estimate.

\begin{proposition}
If
$k^2$ is not a Robin eigenvalue of the interior Laplacian with Robin parameter $\betaparam\D$, 
then
there exists $h_0\in\RR$ such that for all $h<h_0$,
the discrete problem \cref{eq:discrete} has a unique solution.
\end{proposition}
\begin{proof}
From \cref{lemma:AT_coercive,lemma:c_unique}, we see that the assumptions of 
\cite[theorem 4.2.9]{sauter-schwab} are true, hence the result holds.
\end{proof}

To prove our \emph{a priori} error bounds, we will use the following lemma.
\begin{lemma}\label{diri:important_lemma}
Assume that
$\Re(\betaparam\D)>0$,
$h<h_0$ and
$k^2$ is not a Robin eigenvalue of the interior Laplacian with Robin parameter $\betaparam\D$.
Let $(u,\lambda)$ be the solution of \cref{eq:abstract_form_helmDir}, and
let $(u_h,\lambda_h)$ be the solution of \cref{eq:discrete}. These solutions satisfy
\[
\Bnorm[\D]{(u-u_h,\lambda-\lambda_h)} \lesssim
\inf_{(v_h,\mu_h)\in \productspace{V}_h}
\Bnorm[\D]{(u-v_h,\lambda-\mu_h)}.
\]
\end{lemma}
\begin{proof}
This also follows from \cite[theorem 4.2.9]{sauter-schwab}.
\end{proof}

\subsection{\emph{A priori} error bounds}
If $\Gamma$ is a discretisation of a smooth surface, then we can take $\productspace{V}_h=\Poly_h^p(\Gamma)\times\Poly_h^q(\Gamma)$
as our discrete space. The following result gives the approximation properties of this space.

\begin{lemma}[Approximation in $\Poly_h^p(\Gamma)\times\Poly_h^q(\Gamma)$]\label{diri:p1p1_approx}
If $\seminorm{\betaparam\D}\lesssim h^{-1}$,
then $\forall (v,\mu) \in \Hspace{s}\times \Hspace{r}$,
\begin{equation*}
\inf_{(w_h,\eta_h) \in \Poly_h^p(\Gamma)\times\Poly_h^q(\Gamma)}
    \Bnorm[\D]{(v-w_h,\mu - \eta_h)}
\lesssim h^{\zeta-1/2} \Hpwseminorm{\zeta}{v} +
h^{\xi+1/2} \Hpwseminorm{\xi}{\mu},
\end{equation*}
where $\zeta = \min(p+1,s)$, $\xi = \min(q+1,r)$,
$s\geqslant\frac12$ and $r\geqslant-\frac12$.
\end{lemma}
\begin{proof}
This can be proved in the same way as \cite[proposition 4.14]{BeBuScro18}.
\end{proof}

We can now prove the \emph{a priori} error bound for $\productspace{V}_h=\Poly_h^p(\Gamma)\times\Poly_h^q(\Gamma)$.

\begin{theorem}
Assume that
$h<h_0$,
$\Re(\betaparam\D)>0$,
$\seminorm{\betaparam\D}\lesssim h^{-1}$, and
$k^2$ is not a Robin eigenvalue of the interior Laplacian with Robin parameter $\betaparam\D$.
The solution, $(u,\lambda)\in \Hspace{s}\times \Hspace{r}$, of \cref{eq:abstract_form_helmDir}, and
the solution, $(u_h,\lambda_h)\in \Poly_h^p(\Gamma)\times\Poly_h^q(\Gamma)$, of \cref{eq:discrete} satisfy
\[
\Bnorm[\D]{(u-u_h,\lambda-\lambda_h)} \lesssim
h^{\zeta-1/2} \Hpwseminorm{\zeta}{u}
+ h^{\xi+1/2}\Hpwseminorm{\xi}{\lambda},
\]
where $\zeta = \min(p+1,s)$ and $\xi = \min(q+1,r)$.
\end{theorem}
\begin{proof}
Combine \cref{diri:p1p1_approx} and \cref{diri:important_lemma}.
\end{proof}

When $\Gamma$ is not smooth, the flux space must be chosen so that it can represent jumps in the normal
derivative between cells. In this case, we take $\productspace{V}_h=\Poly_h^p(\Gamma)\times\DPoly_h^q(\Gamma)$.
The following result gives the approximation properties of this space.

\begin{lemma}[Approximation in $\Poly_h^p(\Gamma)\times\DPoly_h^q(\Gamma)$]\label{diri:approx}
If $\seminorm{\betaparam\D}\lesssim h^{-1}$,
then $\forall (v,\mu) \in \Hspace{s}\times \Hspace{r}$,
\begin{equation*}
\inf_{(w_h,\eta_h) \in \Poly_h^p(\Gamma)\times\DPoly_h^q(\Gamma)}
    \Bnorm[\D]{(v-w_h,\mu - \eta_h)}
\lesssim h^{\zeta-1/2} \Hpwseminorm{\zeta}{v} +
h^{\xi+1/2} \Hpwseminorm{\xi}{\mu},
\end{equation*}
where $\zeta = \min(p+1,s)$, $\xi = \min(q+1,r)$,
$s\geqslant\frac12$ and $r\geqslant-\frac12$.
\end{lemma}
\begin{proof}
See \cite[proposition 4.14]{BeBuScro18}.
\end{proof}

We can now prove the \emph{a priori} error bound for $\productspace{V}_h=\Poly_h^p(\Gamma)\times\DPoly_h^q(\Gamma)$.

\begin{theorem}
Assume that
$h<h_0$,
$\Re(\betaparam\D)>0$,
$\seminorm{\betaparam\D}\lesssim h^{-1}$, and
$k^2$ is not a Robin eigenvalue of the interior Laplacian with Robin parameter $\betaparam\D$.
The solution, $(u,\lambda)\in \Hspace{s}\times \Hspace{r}$, of \cref{eq:abstract_form_helmDir}, and
the solution, $(u_h,\lambda_h)\in\Poly_h^p(\Gamma)\times\DPoly_h^q(\Gamma)$, of \cref{eq:discrete} satisfy
\[
\Bnorm[\D]{(u-u_h,\lambda-\lambda_h)} \lesssim
h^{\zeta-1/2} \Hpwseminorm{\zeta}{u}
+ h^{\xi+1/2}\Hpwseminorm{\xi}{\lambda},
\]
where $\zeta = \min(p+1,s)$ and $\xi = \min(q+1,r)$.
\end{theorem}
\begin{proof}
Combine \cref{diri:approx} and \cref{diri:important_lemma}.
\end{proof}

Recall that due to \cref{lemma:robin_nontrivial}, the Robin eigenvalue assumption in the results in this
section holds whenvever $\Im(\betaparam\D)\not=0$.

\section{Numerical results}\label{sec:numerical}
In this section, we demonstrate the theory with a series of numerical examples.
All the results presented were computed using version 0.2.4 of Bempp-cl, an open source Python boundary
element method library \cite{bempp-cl}.
We precondition all linear systems in this section with blocked mass matrix preconditioners applied from the left.
We take $\productspace{V}_h=\Poly_h^1(\Gamma)\times\Poly_h^1(\Gamma)$ throughout this section, so we use the preconditioner
\[
\begin{pmatrix}
M&0\\0&M
\end{pmatrix},
\]
where $M=(m_{ij})$ is defined by 
\[
m_{ij}=\Ltwoinner{\phi_j}{\phi_i},
\]
where $\{\phi_0,\phi_1,...\}$ is the basis of the space $\Poly_h^1(\Gamma)$.
The preconditioning corresponds to taking the discrete strong form of the operator, as described in \cite{2020-productalgebras}.

\newcommand{\helmf}[1]{
    \frac{\e^{\ii k\seminorm{#1}}}{\seminorm{#1}}
}
\newcommand{\helmdf}[1]{
    \frac{(\ii k\seminorm{#1}-1)\e^{\ii k\seminorm{#1}}}{\seminorm{#1}^3}#1\cdot\vec\nu
}

Define
\begin{align}
g\D(\x)&=\helmf{\vec{r}_1} + \helmf{\vec{r}_2},\label{helmdirinum_bc}
\end{align}
where $\vec{r}_1=\x-(\frac1{10}, \frac12, \frac12)$ and $\vec{r}_2=\x-(\frac1{10}, \frac14, \frac14)$.
Let $\Omega\interior$ be a bounded open domain such that
$(\frac1{10}, \frac12, \frac12)\in\Omega\interior$ and
$(\frac1{10}, \frac14, \frac14)\in\Omega\interior$, and let
$\Omega\exterior:=\RR^3\setminus\overline{\Omega\interior}$.
It is easy to check that for any wavenumber $k>0$,
$u(\x)=\helmf{\vec{r}_1} + \helmf{\vec{r}_2}$ is the solution of the exterior Helmholtz problem
\begin{subequations}
\begin{align}
-\Delta u-k^2u&=0&&\text{in }\Omega\exterior,\\
\frac{\partial u\scat}{\partial\abs{\x}}-\ii ku\scat&=o(\abs{\x}^{-1})&&\text{as }\abs{\x}\to\infty,\\
u&=g\D&&\text{on }\Gamma,
\end{align}
\end{subequations}
with $u\inc=0$ (and so $u=u\scat$). It can also be seen that
\begin{align}
\frac{\partial u}{\partial\vec\nu}&=\helmdf{\vec{r}_1} + \helmdf{\vec{r}_2}\label{helmneunum_bc}
&&\text{on }\Gamma\N.
\end{align}
Using \cref{helmdirinum_bc,helmneunum_bc}, we can compute the error of the solutions obtained in this section.

\begin{figure}
\centering
\null\hfill
\begin{tikzpicture}
\begin{axis}[width=\plotsize\textwidth,axis on top,
             xlabel={$k$},xmin=1,xmax=5,
             ylabel near ticks,ylabel={\errorlabel},ymode=log,ymax=0.1,ymin=1e-2
]
\referenceplot table {%
2.759 1e-2
2.759 0.1
};
\dotplot table {data/vary_k_real_error.dat};
\end{axis}
\end{tikzpicture}
\hfill
\begin{tikzpicture}
\begin{axis}[width=\plotsize\textwidth,axis on top,
             xlabel={$k$},xmin=1,xmax=5,
             ylabel near ticks,ylabel={{\textnumero} of GMRES iterations},ymin=0,ylabel near ticks,ymax=50
]
\referenceplot table {%
2.759 0
2.759 400
};
\dotplot table {data/vary_k_real_iter.dat};
\end{axis}
\end{tikzpicture}
\hfill\null
\caption{The error (left) and GMRES iteration counts (right) of the penalty method with $\betaparam\D=1$
for the Helmholtz Dirichlet problem with varying $k$ on the unit sphere with $h=2^{-2}$.
Here we take $(u_h,\lambda_h),(v_h,\mu_h)\in\Poly_h^1(\Gamma)\times\Poly_h^1(\Gamma)$
and solve to a GMRES tolerance of $10^{-5}$.}
\label{fig:vary_k_real}
\end{figure}
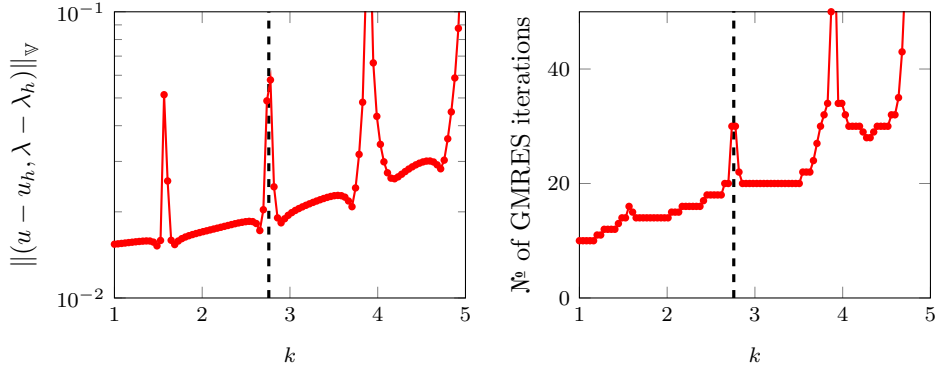

\Cref{fig:vary_k_real} shows how the error and number of GMRES iterations vary as $k$ is increased, using $\betaparam\D=1$.
It can be seen that at various points, the number of iterations and error grows as $k$ approaches a problematic Robin eigenvalue
of the problem: one of these is near 2.759, as indicated by the dashed lines.

\begin{figure}
\centering
\null\hfill
\begin{tikzpicture}
\begin{axis}[width=\plotsize\textwidth,axis on top,
             xlabel={$\betaparam\D$},xmode=log,xmin=1e-6,xmax=1e6,
             ylabel near ticks,ylabel={\errorlabel},ymode=log,ymax=10,ymin=1e-4
]
\referenceplot table {%
1 1e-4
1 10
};
\oneplot table {data/vary_beta_real_a_error.dat};
\sametwoplot table {data/vary_beta_real_b_error.dat};
\samethreeplot table {data/vary_beta_real_c_error.dat};
\end{axis}
\end{tikzpicture}
\hfill
\begin{tikzpicture}
\begin{axis}[width=\plotsize\textwidth,axis on top,
             xlabel={$\betaparam\D$},xmode=log,xmin=1e-6,xmax=1e6,
             ylabel near ticks,ylabel={{\textnumero} of GMRES iterations},ymin=0,ylabel near ticks,ymax=50
]
\referenceplot table {%
1 0
1 400
};
\oneplot table {data/vary_beta_real_a_iter.dat};
\sametwoplot table {data/vary_beta_real_b_iter.dat};
\samethreeplot table {data/vary_beta_real_c_iter.dat};
\end{axis}
\end{tikzpicture}
\hfill\null
\caption{The error (left) and GMRES iteration counts (right) of the penalty method with varying real $\betaparam\D$
for the Helmholtz Dirichlet problem with $k=2.759$ on the unit sphere with
$h=2^{-2}$ (\oneplotdesc),
$h=2^{-3}$ (\sametwoplotdesc),
and $h=2^{-4}$ (\samethreeplotdesc).
Here we take $(u_h,\lambda_h),(v_h,\mu_h)\in\Poly_h^1(\Gamma)\times\Poly_h^1(\Gamma)$
and solve to a GMRES tolerance of $10^{-5}$.}
\label{fig:vary_beta_real}
\end{figure}
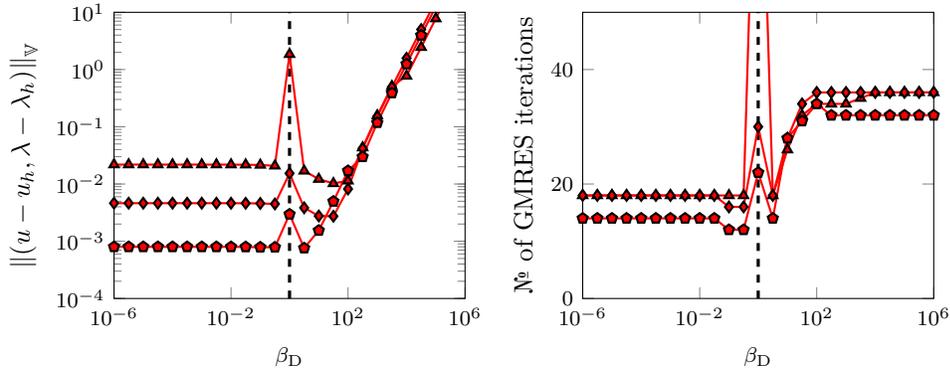

In \cref{sec:analysis}, we saw that our formulation can fail to have a unique solution when $k^2$ is Robin eigenvalues
of the Laplacian with Robin parameter $\betaparam\D$. This implies that adjusting the parameter $\betaparam\D$ will adjust the
locations of these eigenvalues.
This can be oberved in \cref{fig:vary_beta_real}, where the error and number of iterations are shown for
the problem with $k=2.759$ and varying real $\betaparam\D$. The spike in the error and number of iterations at $\betaparam\D=1$
(as shown by the dashed lines)
is due to approaching the same eigenvalue that we saw in \cref{fig:vary_k_real}. We observe that the increase in the error
is less pronounced for meshes with a lower value of $h$. This is due to $2.759$ being near an eigenvalue of a given
discretisation of the sphere: the exact approximation of the sphere changes a little as we increase the number of cells
in the mesh, and so the eigenvalues differ on the different meshes used.

As we saw in \cite{BeBuScro18} for Laplace problems, we observe in \cref{fig:vary_beta_real} that the error and
iteration counts for the problem increase once $\betaparam\D$ is increased above a certain level.

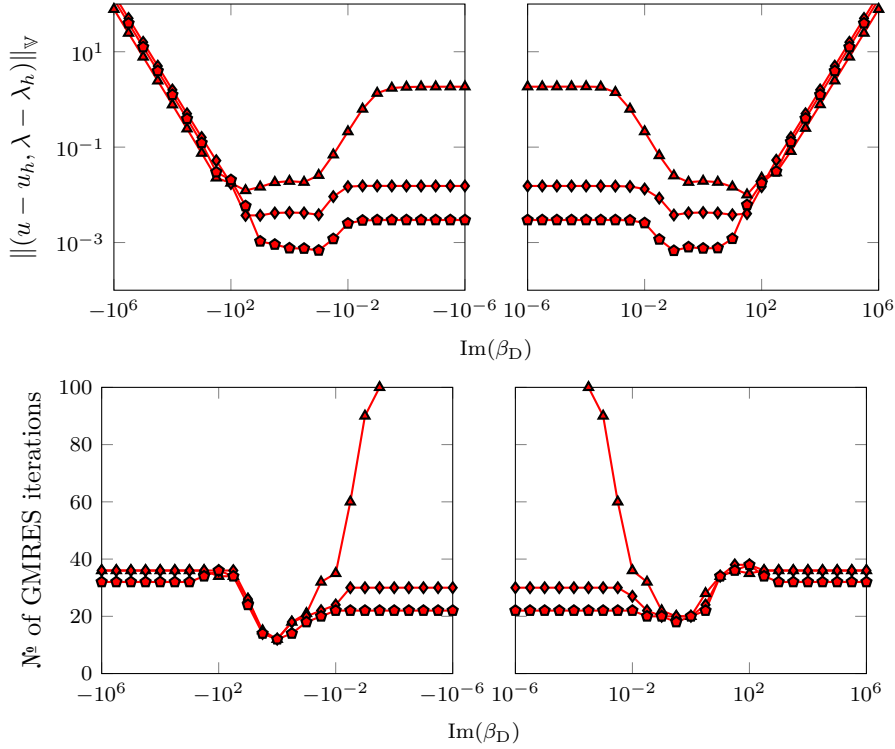
\begin{figure}
\centering
\begin{tikzpicture}
\begin{axis}[width=\plotsize\textwidth,axis on top,
             ,xmin=1e-6,xmax=1e6,xmode=log,x dir=reverse,
             ylabel near ticks,ylabel={\errorlabel},ymode=log,ymax=100,ymin=1e-4,
             xticklabels={$-10^{-6}$,$-10^{-2}$,$-10^{2}$,$-10^{6}$}
]
\oneplot table {data/vary_beta_imag_aneg_error.dat};
\sametwoplot table {data/vary_beta_imag_bneg_error.dat};
\samethreeplot table {data/vary_beta_imag_cneg_error.dat};
\end{axis}
\begin{scope}[shift={(5.5,0)}]
\begin{axis}[width=\plotsize\textwidth,axis on top,
             xlabel={$\Im(\betaparam\D)$},xlabel style={xshift=-2.75cm},xmin=1e-6,xmax=1e6,xmode=log,
             ylabel near ticks,ymode=log,ymax=100,ymin=1e-4,yticklabels={,,}
]
\oneplot table {data/vary_beta_imag_apos_error.dat};
\sametwoplot table {data/vary_beta_imag_bpos_error.dat};
\samethreeplot table {data/vary_beta_imag_cpos_error.dat};
\end{axis}
\end{scope}
\end{tikzpicture}
\\
\begin{tikzpicture}
\begin{axis}[width=\plotsize\textwidth,axis on top,
             xmin=1e-6,xmax=1e6,xmode=log,x dir=reverse,
             ylabel near ticks,ylabel={{\textnumero} of GMRES iterations},ymin=0,ylabel near ticks, ymax=100,
             xticklabels={$-10^{-6}$,$-10^{-2}$,$-10^{2}$,$-10^{6}$}
]
\oneplot table {data/vary_beta_imag_aneg_iter.dat};
\sametwoplot table {data/vary_beta_imag_bneg_iter.dat};
\samethreeplot table {data/vary_beta_imag_cneg_iter.dat};
\end{axis}
\begin{scope}[shift={(5.5,0)}]
\begin{axis}[width=\plotsize\textwidth,axis on top,
             xlabel={$\Im(\betaparam\D)$},xlabel style={xshift=-2.75cm},xmin=1e-6,xmax=1e6,xmode=log,
             ylabel near ticks,ymin=0,ylabel near ticks, ymax=100,yticklabels={,,,,}
]
\oneplot table {data/vary_beta_imag_apos_iter.dat};
\sametwoplot table {data/vary_beta_imag_bpos_iter.dat};
\samethreeplot table {data/vary_beta_imag_cpos_iter.dat};
\end{axis}
\end{scope}
\end{tikzpicture}
\caption{The error (top) and GMRES iteration counts (bottom) of the penalty method with varying $\betaparam\D$ with
$\Re(\betaparam\D)=1$
for the Helmholtz Dirichlet problem with $k=2.759$ on the unit sphere with
$h=2^{-2}$ (\oneplotdesc),
$h=2^{-3}$ (\sametwoplotdesc),
and $h=2^{-4}$ (\samethreeplotdesc).
Here we take $(u_h,\lambda_h),(v_h,\mu_h)\in\Poly_h^1(\Gamma)\times\Poly_h^1(\Gamma)$
and solve to a GMRES tolerance of $10^{-5}$.}
\label{fig:vary_beta_imag}
\end{figure}

\Cref{fig:vary_beta_imag} shows how the error and iteration counts vary as we adjust the imaginary part of $\betaparam\D$
when the real part of $\betaparam\D$ is fixed as 1, again taking $k=2.759$ so that we hit an eigenvalue when
$\Im(\betaparam\D)=0$. We see that once $|\Im(\betaparam\D)|$ is greater than around $10^{-2}$,
the error and iteration count drop. Once $|\Im(\betaparam\D)|$ is too large, the error and iteration count rise in a
similar way to that we observed when taking a large real $\betaparam\D$.
We observe that the iteration count is slightly lower for a small range of values when $\Im(\betaparam\D)$ is negative.

\begin{figure}
\centering
\null\hfill
\begin{tikzpicture}
\begin{axis}[width=\plotsize\textwidth,axis on top,
             xlabel={$k$},xmin=1,xmax=5,
             ylabel near ticks,ylabel={\errorlabel},ymode=log,ymax=0.1,ymin=1e-2
]
\dotplot table {data/vary_k_imag-i_error.dat};
\end{axis}
\end{tikzpicture}
\hfill
\begin{tikzpicture}
\begin{axis}[width=\plotsize\textwidth,axis on top,
             xlabel={$k$},xmin=1,xmax=5,
             ylabel near ticks,ylabel={{\textnumero} of GMRES iterations},ymin=0,ylabel near ticks,ymax=50
]
\dotplot table {data/vary_k_imag-i_iter.dat};
\end{axis}
\end{tikzpicture}
\hfill\null
\caption{The error (left) and GMRES iteration counts (right) of the penalty method with $\betaparam\D=1-\ii$
for the Helmholtz Dirichlet problem with varying $k$ on the unit sphere with $h=2^{-2}$.
Here we take $(u_h,\lambda_h),(v_h,\mu_h)\in\Poly_h^1(\Gamma)\times\Poly_h^1(\Gamma)$
and solve to a GMRES tolerance of $10^{-5}$.}
\label{fig:vary_k_real-i}
\end{figure}

Motivated by our observations in \cref{fig:vary_beta_imag}, we fix $\betaparam\D=1-\ii$. 
\Cref{fig:vary_k_real-i} shows how the error and iteration counts change as we increase $k$ with this value of
$\betaparam\D$. In agreement with \cref{lemma:robin_nontrivial},
we observe that (in contrast to \cref{fig:vary_k_real}) there is no vulnerability to
eigenvalues in this case, and the iteration count remains steady as we increase $k$.

\begin{figure}
\centering
\null\hfill
\begin{tikzpicture}
\begin{axis}[width=\plotsize\textwidth,axis on top,axis equal,
             xlabel={$h$},x dir=reverse,xmode=log,
             ylabel near ticks,ylabel={\errorlabel},ymode=log,
]
\oneplot table {data/convergence_3_error.dat};
\referenceplot table {%
0.5 0.25
0.02 0.0004
};
\end{axis}
\end{tikzpicture}
\hfill
\begin{tikzpicture}
\begin{axis}[width=\plotsize\textwidth,axis on top,
             xlabel={$h$},x dir=reverse,xmode=log,xmin=3e-2,xmax=1,
             ylabel near ticks,ylabel={{\textnumero} of GMRES iterations},ymin=0,ylabel near ticks,ymax=50
]
\oneplot table {data/convergence_3_iter.dat};
\end{axis}
\end{tikzpicture}
\hfill\null
\caption{The error (left) and GMRES iteration counts (right) of the penalty method with $\betaparam\D=1-\ii$
for the Helmholtz Dirichlet problem with $k=3$ on the unit sphere as we reduce $h$.
The dashed line shows order 2 convergence.
Here we take $(u_h,\lambda_h),(v_h,\mu_h)\in\Poly_h^1(\Gamma)\times\Poly_h^1(\Gamma)$
and solve to a GMRES tolerance of $10^{-5}$.}
\label{fig:convergence}
\end{figure}
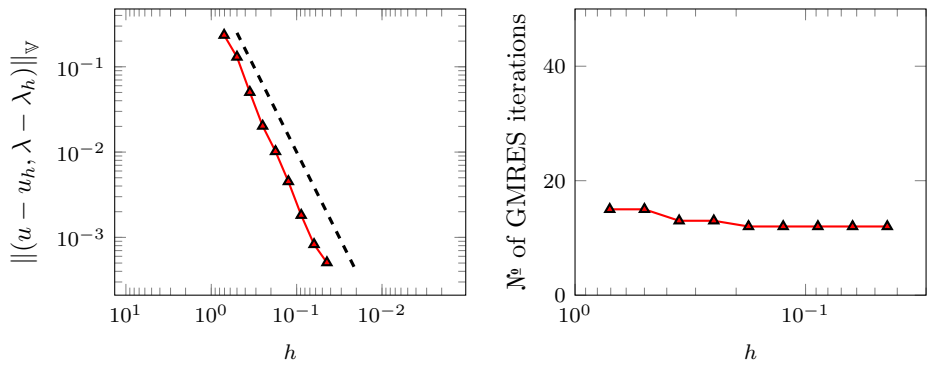

\Cref{fig:convergence} shows the error and number of iterations as we reduce $h$. We observe order 2 convergence,
and see that the number of iterations remains the same as $h$ is reduced, demonstrating the effectiveness of the
mass matrix preconditioner.

\section{Conclusions}\label{sec:conclusions}
In this paper, we have derived and analysed a formulation for weak imposition of Dirichlet boundary conditions on the Helmholtz
equation. By taking a parameter with a non-zero imaginary part, Helmholtz problems can be solved at any real wavenumber without
any difficulties caused by resonances of the interior problem.

\label{sec:mixed}
\begin{figure}
\centering
\begin{tikzpicture}
     \node at (0,0) {\includegraphics[width=90mm]{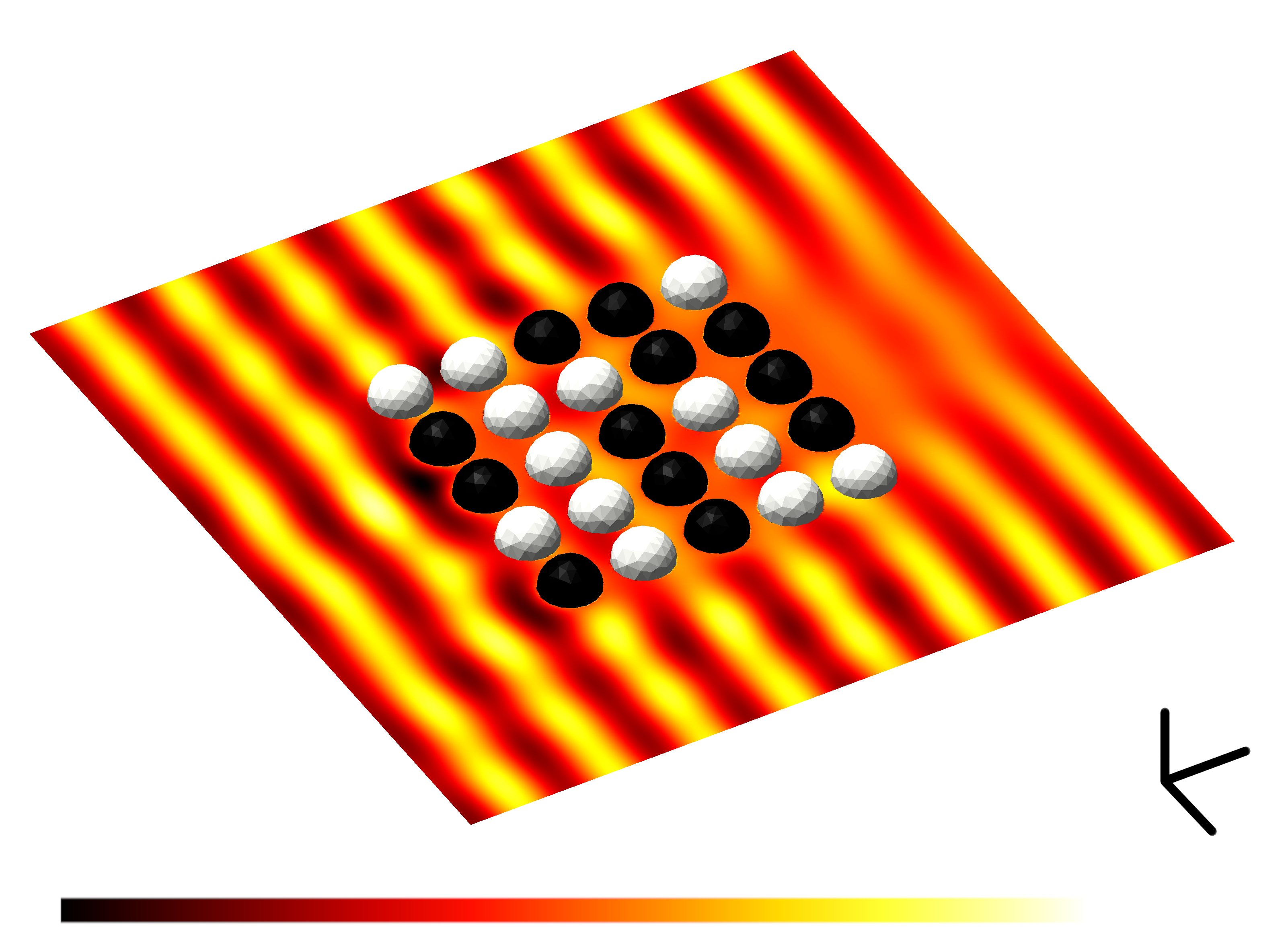}};
    \begin{footnotesize}
     \node at (4.1,-2.35) {$x$};
     \node at (4.35,-1.8) {$y$};
     \node at (3.7,-1.5) {$z$};
     \node at (-4.1,-2.75) {$-1.9$};
     \node at (-0.5,-2.75) {$0$};
     \node at (3.1,-2.75) {$1.9$};
     \node at (0,-4) {\ };
    \end{footnotesize}
\end{tikzpicture}
\caption{The incident wave $u\inc=\e^{\ii k\x\cdot\vec{d}}$, where $\vec{d}=(0,1,0)$ and $k=2$, scattering off
25 spheres. The white spheres are sound-hard and the black spheres are sound-soft.}
\label{fig:helmmanyplot}
\end{figure}
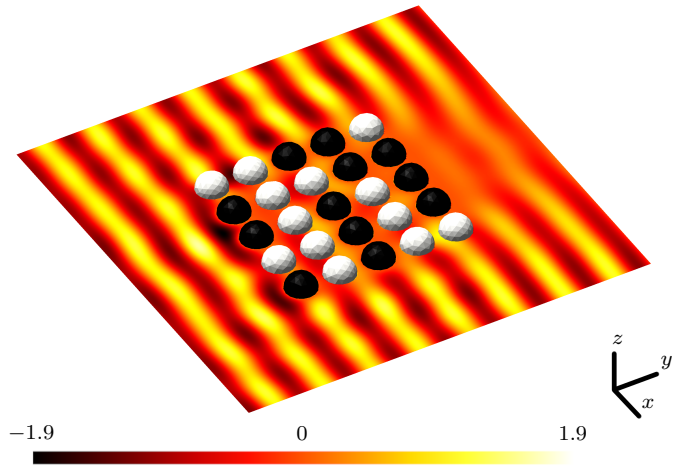

The formulation derived in this paper bears a close resemblance to the formulations for Laplace that we
derived and analysed in \cite{BeBuScro18}. Formulations for Helmholtz problems with mixed Dirichlet--Neumann or Robin
boundary conditions could be derived in the same way. We expect that these formulations could be analysed following a
similar method as used here, although this analysis appears to pose some additional challenges.
The weak imposition of mixed boundary conditions is demonstrated in \cref{fig:helmmanyplot}, where
we have plotted the scattering of an incident wave colliding with a collection of sound-hard
(a 0 Neumann boundary condition)
and sound-soft (a 0 Dirichlet boundary condition) spheres.

One benefit of formulating mixed problems in this way is that the boundary conditions are imposed
by adding sparse terms to the full Calder\'on system: for mixed boundary conditions, sparse terms assembled
on parts of the boundary can be added without any need to adjust the dense Calder\'on term. When solving an inverse
problem---for example when looking to find the material properties that should be used to give a scatterer a
certain desired property---the Calder\'on term (which is the most expensive part to assemble) can be reused
and different sparse terms added to solve the same problem with different boundary conditions.

One avenue of interest for further interest would be the weak imposition of boundary condition on Maxwell problems.
In the experiments we have run to explore this, however, we have been unable to obtain good solutions in a reasonable
amount of time. Maxwell problems are prone to being strongly ill-conditioned, and it appears that mass matrix preconditioning
is not enough to achieve good performance in the Maxwell case.
Therefore, we believe that it is necessary to design more powerful preconditioners for these weak formulations
in order to make this method feasible for Maxwell problems.

\section*{Acknowledgements} Erik Burman and Timo Betcke were supported by EPSRC
grants EP/P01576X/1 and EP/P012434/1.
We would like to thank the reviewers of this paper for their helpful comments; in particular, their insightful comments that enabled us to simplify many parts of the analysis
in \cref{sec:analysis}.

\bibliographystyle{siamplain}
\bibliography{references}

\begin{thebibliography}{10}

\bibitem{Babuska1973}
{\sc I.~Babu\u{s}ka}, {\em The finite element method with penalty}, Mathematics
  of Computation, 27 (1973), pp.~221--228,
  \url{https://doi.org/10.2307/2005611}.

\bibitem{Banjai07}
{\sc L.~Banjai and S.~Sauter}, {\em A refined {G}alerkin error and stability
  analysis for highly indefinite variational problems}, SIAM Journal on
  Numerical Analysis, 45 (2007), pp.~{37--53},
  \url{https://doi.org/10.1137/060654177}.

\bibitem{BeBuScro18}
{\sc T.~Betcke, E.~Burman, and M.~W. Scroggs}, {\em Boundary element methods
  with weakly imposed boundary conditions}, SIAM Journal on Scientific
  Computing, 41 (2019), pp.~A1357--A1384,
  \url{https://doi.org/10.1137/18M119625X}.

\bibitem{bempp-cl}
{\sc T.~Betcke and M.~W. Scroggs}, {\em Bempp-cl: {A} fast {P}ython based
  just-in-time compiling boundary element library}, Journal of Open Source
  Software,  (2021), p.~{2879}, \url{https://doi.org/10.21105/joss.02879}.

\bibitem{2020-productalgebras}
{\sc T.~Betcke, M.~W. Scroggs, and W.~\'{S}migaj}, {\em Product algebras for
  {G}alerkin discretisations of boundary integral operators and their
  applications}, ACM Transactions on Mathematical Software, 46 (2020),
  pp.~{4:1--4:22}, \url{https://doi.org/10.1145/3368618}.

\bibitem{brezis-book}
{\sc H.~Brezis}, {\em Functional Analysis, Sobolev Spaces and Partial
  Differential Equations}, Springer, 2011,
  \url{https://doi.org/10.1007/978-0-387-70914-7}.

\bibitem{BuFrScro19}
{\sc E.~Burman, S.~Frei, and M.~W. Scroggs}, {\em Weak imposition of
  {S}ignorini boundary conditions on the boundary element method}, SIAM Journal
  on Numerical Analysis, 58 (2020), pp.~{2334--2350},
  \url{https://doi.org/10.1137/19M1281721}.

\bibitem{Graham07}
{\sc V.~Dominguez, I.~Graham, and V.~Smyshlyaev}, {\em A hybrid
  numerical-asymptotic boundary integral method for high-frequency acoustic
  scattering}, Numerische Mathematik, 106 (2007), pp.~{471--510},
  \url{https://doi.org/10.1007/s00211-007-0071-4}.

\bibitem{Steinbach08}
{\sc S.~Engleder and O.~Steinbach}, {\em Stabilized boundary element methods
  for exterior {H}elmholtz problems}, Numerische Mathematik, 110 (2008),
  pp.~{145--160}, \url{https://doi.org/10.1007/s00211-008-0161-y}.

\bibitem{Graham15}
{\sc I.~Graham, M.~L{\"o}hndorf, J.~Melenk, and E.~Spence}, {\em When is the
  error in the $h$-{BEM} for solving the {H}elmholtz equation bounded
  independently of $k$?}, BIT Numerical Mathematics, 55 (2015), pp.~{171--214},
  \url{https://doi.org/10.1007/s10543-014-0501-5}.

\bibitem{Hiptmair06}
{\sc R.~Hiptmair and P.~Meury}, {\em Stabilized {FEM}-{BEM} coupling for
  {H}elmholtz transmission problems}, SIAM Journal on Numerical Analysis, 44
  (2006), pp.~{2107--2130}, \url{https://doi.org/10.1137/050639958}.

\bibitem{Sayas008}
{\sc A.~Laliena, M.-L. Rapún, and F.-J. Sayas}, {\em Symmetric boundary
  integral formulations for {H}elmholtz transmission problems}, Applied
  Numerical Mathematics, 59 (2009), pp.~{2814--2823},
  \url{https://doi.org/10.1016/j.apnum.2008.12.030}.

\bibitem{Melenk11-2}
{\sc M.~Löhndorf and J.~M. Melenk}, {\em Wavenumber-explicit $hp$-{BEM} for
  high frequency scattering}, SIAM Journal on Numerical Analysis, 49 (2011),
  pp.~{2340--2363}, \url{https://doi.org/10.1137/100786034}.

\bibitem{Melenk12}
{\sc J.~M. Melenk}, {\em Mapping properties of combined field {H}elmholtz
  boundary integral operators}, SIAM Journal on Mathematical Analysis, 44
  (2012), pp.~{2599--2636}, \url{https://doi.org/10.1137/100784072}.

\bibitem{Nedelec01}
{\sc J.-C. N\'ed\'elec}, {\em Acoustic and Electromagnetic Equations}, vol.~144
  of Applied Mathematical Sciences, Springer-Verlag New York, 2001,
  \url{https://doi.org/10.1007/978-1-4757-4393-7}.

\bibitem{Nit71}
{\sc J.~Nitsche}, {\em \"{U}ber ein {V}ariationsprinzip zur {L}\"osung von
  {D}irichlet-{P}roblemen bei {V}erwendung von {T}eilr\"aumen, die keinen
  {R}andbedingungen unterworfen sind}, Abhandlungen aus dem Mathematischen
  Seminar der Universit\"at Hamburg, 36 (1971), pp.~9--15,
  \url{https://doi.org/10.1007/BF02995904}.
\newblock Collection of articles dedicated to Lothar Collatz on his sixtieth
  birthday.

\bibitem{Sayas08}
{\sc M.-L. Rapún and F.-J. Sayas}, {\em Mixed boundary integral methods for
  {H}elmholtz transmission problems}, Journal of Computational and Applied
  Mathematics, 214 (2008), pp.~{238--258},
  \url{https://doi.org/10.1016/j.cam.2007.02.028}.

\bibitem{sauter-schwab}
{\sc S.~A. Sauter and C.~Schwab}, {\em Boundary Element Methods}, vol.~39 of
  Springer Series in Computations Mathematics, Springer-Verlag, 2011,
  \url{https://doi.org/10.1007/978-3-540-68093-2}.

\bibitem{Schatz74}
{\sc A.~H. Schatz}, {\em An observation concerning {R}itz--{G}alerkin methods
  with indefinite bilinear forms}, Mathematics of Computation, 28 (1974),
  pp.~{959--962}, \url{https://doi.org/10.1090/S0025-5718-1974-0373326-0}.

\bibitem{Sommerfeld}
{\sc S.~H. Schot}, {\em Eighty years of {S}ommerfeld's radiation condition},
  Historia Mathematica, 19 (1992), pp.~{385--401},
  \url{https://doi.org/10.1016/0315-0860(92)90004-U}.

\bibitem{Spence11}
{\sc E.~A. Spence, S.~N. Chandler-Wilde, I.~G. Graham, and V.~P. Smyshlyaev},
  {\em A new frequency-uniform coercive boundary integral equation for acoustic
  scattering}, Communications on Pure and Applied Mathematics, 64 (2011),
  pp.~{1384--1415}, \url{https://doi.org/10.1002/cpa.20378}.

\bibitem{Spence15}
{\sc E.~A. Spence, I.~V. Kamotski, and V.~P. Smyshlyaev}, {\em Coercivity of
  combined boundary integral equations in high-frequency scattering},
  Communications on Pure and Applied Mathematics, 68 (2015), pp.~{1587--1639},
  \url{https://doi.org/10.1002/cpa.21543}.

\bibitem{Stein07}
{\sc O.~Steinbach}, {\em Numerical approximation methods for elliptic boundary
  value problems}, Springer-Verlag, 2008,
  \url{https://doi.org/10.1007/978-0-387-68805-3}.
\newblock Translated from the 2003 {G}erman original.

\end{thebibliography}
\end{document}